\documentclass[12pt]{amsart}
\usepackage{epsfig}
\usepackage{graphicx}
\usepackage{amscd}
\usepackage{amsmath}
\usepackage{amsxtra}
\usepackage{amsfonts}
\usepackage{amssymb}

\newtheorem{theorem}{Theorem}[section]
\newtheorem{corollary}[theorem]{Corollary}
\newtheorem{lemma}[theorem]{Lemma}
\newtheorem{proposition}[theorem]{Proposition}

\newtheorem{conjecture}[theorem]{Conjecture}
\theoremstyle{definition}
\newtheorem{definition}[theorem]{Definition}
\newtheorem{remark}[theorem]{Remark}

\newtheorem{example}[theorem]{Example}
\theoremstyle{remark}

\renewcommand{\theclaim}{\textup{\theclaim}}

\numberwithin{equation}{section}

\def\openone

{\mathchoice

{\hbox{\upshape \small1\kern-3.3pt\normalsize1}}

{\hbox{\upshape \small1\kern-3.3pt\normalsize1}}

{\hbox{\upshape \tiny1\kern-2.3pt\SMALL1}}

{\hbox{\upshape \Tiny1\kern-2pt\tiny1}}}

\makeatletter

\newbox\ipbox

\newcommand{\diracb}[1]{\left\langle #1\mathrel{\mathchoice

{\setbox\ipbox=\hbox{$\displaystyle \left\langle\mathstrut
#1\right.$}

\vrule height\ht\ipbox width0.25pt depth\dp\ipbox}

{\setbox\ipbox=\hbox{$\textstyle \left\langle\mathstrut
#1\right.$}

\vrule height\ht\ipbox width0.25pt depth\dp\ipbox}

{\setbox\ipbox=\hbox{$\scriptstyle \left\langle\mathstrut
#1\right.$}

\vrule height\ht\ipbox width0.25pt depth\dp\ipbox}

{\setbox\ipbox=\hbox{$\scriptscriptstyle \left\langle\mathstrut
#1\right.$}

\vrule height\ht\ipbox width0.25pt depth\dp\ipbox}

}\right. }

\newcommand{\dirack}[1]{\left. \mathrel{\mathchoice

{\setbox\ipbox=\hbox{$\displaystyle \left.\mathstrut
#1\right\rangle$}

\vrule height\ht\ipbox width0.25pt depth\dp\ipbox}

{\setbox\ipbox=\hbox{$\textstyle \left.\mathstrut
#1\right\rangle$}

\vrule height\ht\ipbox width0.25pt depth\dp\ipbox}

{\setbox\ipbox=\hbox{$\scriptstyle \left.\mathstrut
#1\right\rangle$}

\vrule height\ht\ipbox width0.25pt depth\dp\ipbox}

{\setbox\ipbox=\hbox{$\scriptscriptstyle \left.\mathstrut
#1\right\rangle$}

\vrule height\ht\ipbox width0.25pt depth\dp\ipbox}

} #1\right\rangle}

\newcommand{\bz}{\mathbb{Z}}

\newcommand{\B}{\mathcal{B}}
\newcommand{\esssup}{\operatorname*{esssup}}
\newcommand{\essinf}{\operatorname*{essinf}}
\newcommand{\br}{\mathbb{R}}
\newcommand{\bc}{\mathbb{C}}

\newcommand{\bn}{\mathbb{N}}

\def\blfootnote{\xdef\@thefnmark{}\@footnotetext}

\input xy
\xyoption{all}
\usepackage{amssymb}

\def\H{\mathcal{H}}

\def\-{^{-1}}
\def\B{\mathcal{B}}
\def\D{\mathcal{D}}

\begin{document}

\title[Uniformity of measures with Fourier frames]{Uniformity of measures with Fourier frames}
\author{Dorin Ervin Dutkay}
\blfootnote{}
\address{[Dorin Ervin Dutkay] University of Central Florida\\
	Department of Mathematics\\
	4000 Central Florida Blvd.\\
	P.O. Box 161364\\
	Orlando, FL 32816-1364\\
U.S.A.\\} \email{Dorin.Dutkay@ucf.edu}

\author{Chun-Kit Lai}
\address{Department of Mathematics, The Chinese University of Hong Kong,  Hong Kong}
\email{cklai@math.cuhk.edu.hk}

\thanks{}

\begin{abstract}
We examine Fourier frames and, more generally, frame measures for different probability measures. We prove that if a measure has an associated frame measure, then it must have a certain uniformity in the sense that the weight is distributed quite uniformly on its support. To be more precise, by considering certain absolute continuity properties of the measure and its translation, we recover the characterization on absolutely continuous measures $g\, dx$ with Fourier frames obtained in \cite{Lai11}. Moreover, we prove that the frame bounds are pushed away by the essential infimum and supremum of the function $g$. This also shows that absolutely continuous spectral measures supported on a set $\Omega$, if they exist, must be the standard Lebesgue measure on $\Omega$ up to a multiplicative constant. We then investigate affine iterated function systems (IFSs), we show that if an IFS with no overlap admits a frame measure then the probability weights are all equal. Moreover, we also show that the {\L}aba-Wang conjecture \cite{MR1929508} is true if the self-similar measure is absolutely continuous. Finally, we will present a new approach to the conjecture of Liu and Wang \cite{LW} about the structure of non-uniform Gabor orthonormal bases of the form ${\mathcal G}(g,\Lambda,{\mathcal J})$.
\end{abstract}
\maketitle \tableofcontents

\section{Introduction}
Everyone knows about Fourier series: the exponential functions $\{e^{2\pi i nx} : n\in\bz\}$ form an orthonormal basis for $L^2[0,1]$. Perturbations of the set $\bz$ will produce frames for $L^2[0,1]$, or ``non-harmonic'' Fourier series, see e.g., \cite{DS52a, OSANN}. This idea was later extended to orthonormal bases or frames of exponentials (Fourier frames) for fractal measures \cite{DHS09,HL08,
JP98,DHS09,HL08,MR1744572,MR2338387,MR2200934,MR2297038,MR1785282,MR2279556,MR2443273,DHSW10,DHW11a,HLL11}.

In \cite{DHW11b} the notion of frames of exponentials for an arbitrary measure was extended to that of a frame measure.

\begin{definition}\label{def0.1}
Let $\mu$ be a finite, compactly supported Borel measure on
$\br^d$. The {\it Fourier transform} of a function $f\in L^1(\mu)$
is defined by
$$\widehat{f\,d\mu}(t)=\int f(t)e^{-2\pi i t\cdot x}\,d\mu(x),\quad(t\in\br^d).$$

Denote by $e_t$, $t\in\br^d$, the exponential function
$$e_t(x)=e^{2\pi i t\cdot x},\quad(x\in\br^d).$$

 We say that a Borel measure $\nu$ is a {\it Bessel measure} for $\mu$ if there exists a constant $B>0$ such that for every $f \in L^2(\mu)$, we have
\[ \|\widehat{f\,d\mu} \|_{L^2(\nu)}^{2}\leq B \| f\|_{L^2(\mu)}^{2} . \]
We call $B$ a {\it (Bessel) bound} for $\nu$. We say the measure
$\nu$ is a {\it frame measure} for $\mu$ if there exists constants
$A,B > 0$ such that for every $f \in L^2(\mu)$, we have
\[ A\| f\|_{L^2(\mu)}^{2}\leq \|\widehat{f\,d\mu} \|_{L^2(\nu)}^{2} \leq B \| f\|_{L^2(\mu)}^{2}. \]
We call $A,B$ {\it (frame) bounds} for $\nu$. We call $\nu$ a {\it tight frame measure} if $A=B$ and {\it
Plancherel measure} if $A=B=1$.
\end{definition}

 Using the above definitions, we see that a set $E(\Lambda):=\{e_\lambda : \lambda\in\Lambda\}$ is a Fourier frame for $L^2(\mu)$ if and only if the measure $\nu=\sum_{\lambda\in\Lambda}\delta_\lambda$ is a frame measure for $\mu$. $\{e_\lambda : \lambda\in\Lambda\}$ is a tight frame if and only if the measure $\nu=\sum_{\lambda\in\Lambda}\delta_\lambda$ is a tight frame measure for $\mu$. When $E(\Lambda)$ is an orthonormal bases, $\mu$ is called a {\it spectral measure} and $\Lambda$ is called a {\it spectrum} of $\mu$ (\cite{JP98,MR1929508}).

\medskip

In \cite{Lai11}, Lai proved that for absolutely continuous measures $d\mu=g(x)\,dx$, if there exists a Fourier frame, then the function $g$ must be bounded above and below on its support. The proof is based on comparing the Beurling densities. In this paper, we give another approach to prove the theorem.  The main idea is to consider the translates of the original measure $\mu$ restricted to some subset $F$ with $\mu(F)>0$. We denote here by $\omega$ the measure $\omega (\cdot) = T_a\mu|_{F+a}(\cdot) =\mu ((\cdot+a)\cap(F+a))$ with $a\in{\Bbb R}^d$ (see section 2 for details). We have the following theorem.

  \begin{theorem}\label{th0.1+}
Let $\mu$ be a finite Borel measure on $\br^d$ and suppose there exists a frame measure for $\mu$, with frame bounds $A,B>0$. Assume $\omega\ll\mu$. Then
$$
\frac{B}{A}\geq\left\|\frac{d\omega}{d\mu}\right\|_\infty.
$$
  \end{theorem}

  This result shows that the frame bounds control the change of the measure along translations. It will be the key step in this paper and it will work also for other general measures which satisfy this translational absolute continuity assumption, not just the Lebesgue measure. First, we will extend the result in \cite{Lai11} by showing that the essential supremum and infimum of the function $g$ will push away the frame bounds of any frame measure for $d\mu=g\,dx$. In particular, if $g$ is not bounded below or above on its support, then no such frame measure can exist.

 \begin{theorem}\label{th0.1}
Let $d\mu=g\,dx$ be an absolutely continuous measure on $\br^d$. If $\nu$ is a frame measure for $\mu$ with frame bounds $A,B>0$ then
$$\frac BA\geq \frac{\esssup_\mu(g)}{\essinf_\mu(g)}.$$
\end{theorem}

 \medskip

It has been conjectured that a spectral measure must be uniform on its support. It is known that for discrete measures, spectral measures must have only finitely many atoms and the atoms must have equal weight (\cite{MR2200934,HLL11}). For absolutely continuous measures, spectral measures on finite union of intervals must have uniform density (\cite{MR2200934,DHJ09}). Now, an immediate corollary to the inequality in Theorem \ref{th0.1} is the complete solution to this problem in the case of absolutely continuous {\it spectral} measures. More generally, we have

 \begin{corollary}\label{th0.2}
In the hypotheses of Theorem \ref{th0.1} suppose $\mu=g\,dx$  admits a tight frame measure. Then $g$ is a characteristic function of its support.
\end{corollary}

\medskip

For the case singular measures, the conjecture on spectral self-similar measures of {\L}aba and Wang in \cite{MR1929508} asserts that these spectral measures occur only for equal probability weights and when the digit set $\B$ has a tiling property. In the following, we consider the invariant measure associated to an affine iterated function system:
$$
\mu_{\B} = \sum_{b\in \B}p_b\mu_{\B}\circ\tau_{b}^{-1},
$$
where $\tau_b(x) =R^{-1}(x+b)$. Assuming also the \textit{no overlap} condition for $\mu_{\B}$ (i.e. $\mu_{\B}(\tau_b(X_B)\cap\tau_{b'}(X_B))=0$, where $X_\B$ is the {\it attractor} of the IFS) and checking the translational absolute continuity assumption in Theorem \ref{th1.2}, we prove the following result.

 \begin{theorem}\label{th0.3-}
If $\mu_{\B}$ defined above satisfies the no overlap condition and $\mu_{\B}$ admits a frame measure, then all $p_b$ must be equal.
\end{theorem}

If the affine iterated function system does not satisfy the no overlap condition, it is not known whether we still have the above conclusion. However,  with a freedom of choosing the probability weights and the maps, it is of interest to investigate the existence of frame measures in this case. We found that the frame bounds, probability weights and the contraction ratio are closely related. In particular, we can solve the {\L}aba-Wang conjecture when the self-similar measures is absolutely continuous.

\begin{theorem}\label{th0.4-}
Suppose $\mu$ defined in (\ref{eq4.1.2}) is absolutely continuous with respect to the Lebesgue measure and suppose $\mu$ admits a tight frame measure. Then

\vspace{0.2 cm} {\rm (i)} $p_1=\cdots=p_N=\lambda$.

\vspace{0.2 cm} {\rm (ii)} $\lambda=\frac{1}{N}$.

\vspace{0.2 cm} {\rm (iii)} There exists $\alpha>0$ such that $\D: = \alpha\B\subset{\Bbb Z}$ and $\D$ tiles ${\Bbb Z}$.

\end{theorem}

 To formulate this in another way, Theorem \ref{th0.4-} shows that the only absolutely continuous self-similar measures admitting exponential orthonormal bases/ tight frames/ tight frame measures are the measures supported on a self-similar tile by (ii) and \cite{LgW1}. The statement in (iii) says that tile digit set $\D$ will be a scaled integer tile. This is proved by considering the self-replicating tiling set of the attractor $X_\B$.

 Our study is based on the translational absolute continuity assumption. We were not able to show that measures with frame measures must have always this property. But from all the examples that we have, this conjecture should be true. Nonetheless, we can construct examples of singular measures for which the translational absolute continuity assumption in Theorem \ref{th0.1+} fails.

\bigskip

Our results on frame measures and spectral measures also have applications to {\it Gabor systems} (also known as {\it Weyl-Heisenberg systems}). Given $g\in L^2({\Bbb R}^d)$ and a discrete set $\Gamma\in{\Bbb R}^{2d}$, a \textit{Gabor system} is a set of functions:
$$
{\mathcal G}(g,\Gamma) = \{e^{2\pi i  a\cdot x}g(x-b): a,b\in{\Bbb R^d } \ {\mbox{and}} \ (a,b)\in\Gamma\}.
$$
Such a system is called a \textit{Gabor frame (Gabor orthonormal basis)} if ${\mathcal G}(g,\Gamma)$ is a frame (an orthonormal basis) on $L^2({\Bbb R}^d)$. If $\Gamma = \Lambda\times{\mathcal J}$, we will write ${\mathcal G}(g,\Lambda,{\mathcal J}) = {\mathcal G}(g,\Lambda\times{\mathcal J})$. Basic theory of Gabor systems can be found in \cite{G00} and the references therein.

\medskip

In \cite{LW}, the function $g = ({\mathcal L}(\Omega))^{-1/2}\chi_{\Omega}$ (${\mathcal L}$ denotes the Lebesgue measure) with $\Lambda$ and ${\mathcal J}$ discrete subsets of ${\Bbb R}^d$ were considered and the following proposition is proved:

\begin{proposition}\label{th0.3}
\cite{LW}  Suppose that

{\rm (i) $|g| = ({\mathcal L}(\Omega))^{-1/2}\chi_{\Omega}$ where $\Omega$ is a bounded measurable set.}

{\rm (ii)} $\{e_{\lambda}:\lambda\in\Lambda\}$ is an orthonormal basis of $L^2(\Omega)$ and

{\rm (iii)} ${\mathcal J}$ is a tiling set of $\Omega$.

\noindent Then ${\mathcal G}(g,\Lambda,{\mathcal J})$ is a Gabor orthonormal basis of $L^2({\Bbb R}^d)$.
\end{proposition}
${\mathcal J}$ is a \textit{tiling set} of $\Omega$ means that $\bigcup_{t\in{\mathcal J}}(\Omega+t)$ covers ${\Bbb R^d}$ and the intersection of $\Omega+t$ and $\Omega+t'$ has zero Lebesgue measure for distinct $t$ and $t'$. In this case, ${\Omega}$ is a \textit{translational tile}. The proof of this proposition is  a standard generalization of the proof that ${\mathcal G}(\chi_{[0,1]},{\Bbb Z},{\Bbb Z})$ is a Gabor orthonormal basis.

In the literature on Gabor systems, there are many examples of functions $g$ that form a Gabor frame with some $\Gamma$. For example, if $g$ is a compactly supported function with $|g(x)|\geq c>0$ on some small cube, then there exists a $\Gamma$ so that ${\mathcal G}(g,\Gamma)$ is a Gabor frame (see [Gro00, p.125]). However, the requirement for orthonormal bases is more restrictive. There is no known example of a function $g$ which is not a characteristic function such that its associated Gabor system forms an orthonormal basis with some $\Gamma$. Therefore, Liu and Wang conjectured that the converse of the above proposition holds and they proved it in the case when $g$ is supported on an interval. In the following, we prove

\begin{theorem}\label{th0.4}
If the window function $g$ is non-negative, the converse of Proposition \ref{th0.3} holds.
\end{theorem}


\medskip

We organize the paper as follows: we will prove Theorem \ref{th0.1+} and Theorem \ref{th0.1} in section 2 as Theorem \ref{th1.2} and \ref{th1.3} respectively. Then we give a discussion of the corollaries of Theorem \ref{th0.1}, in particular, Corollary \ref{th0.2}. In section 3, we consider the affine iterated function system and prove Theorem \ref{th0.3-}. In section 4, we investigate the iterated function system on ${\Bbb R}^1$ and prove Theorem \ref{th0.4-}. In section 5, we present some concluding remarks on frame measures for singular measures. In the last section, we will focus on Gabor orthonormal bases and prove Theorem \ref{th0.4}.

\medskip

\section{Frame measures}

\begin{definition}\label{def1.1}
Let $\mu$ be a Borel measure on $\br^d$. For a Borel subset $E$ of $\br^d$, we denote by $\mu|_E$ the restriction of $\mu$ to the set $E$, i.e.,
$$\mu|_E(F):=\mu(E\cap F),\mbox{ for all Borel subsets }F\mbox{ of }\br^d.$$

For $a\in\br^d$, we denote by $T_a\mu$, the translation by $a$ of the measure $\mu$, i.e.,
$$T_a\mu(F):=\mu(F+a),\mbox{ for all Borel subsets }F\mbox{ of }\br^d.$$
This means that
$$\int f\,dT_a\mu=\int f(x-a)\,d\mu(x)$$
for all functions $f\in L^1(T_a\mu)$.
\end{definition}

 Throughout the paper, we will use the standard notation $\mu\ll\nu$ to indicate that $\mu$ is absolutely continuous with respect to $\nu$ and we use the notation $\frac{d\mu}{d\nu}$ for its Radon-Nikodym derivative if $\nu$ is $\sigma-$finite. The following theorem is the key step for the next results in this paper.
\begin{theorem}\label{th1.2}
Let $\mu$ be a finite Borel measure on $\br^d$ and suppose there exists a frame measure $\nu$ for $\mu$, with frame bounds $A,B>0$. Assume in addition that there exists a set $F$ of positive measure $\mu$ and $a\in\br^d$ such that the measure $T_a(\mu|_{F+a})\ll\mu$. Then
\begin{equation}
\frac{B}{A}\geq\left\|\frac{dT_a(\mu|_{F+a})}{d\mu}\right\|_\infty.
\label{eq1.2.1}
\end{equation}

\end{theorem}

\begin{proof}
Let $h:=\frac{dT_a(\mu|_{F+a})}{d\mu}$ and let $M:=\|h\|_\infty$. Of course, if $M<1$ there is nothing to prove, so we can assume $M\geq1$. Restricting to a subset of $F$ we can assume also $M<\infty$. We have for bounded functions $f$:
$$\int f\,dT_a\mu|_{F+a}=\int f(x-a)\,d\mu|_{F+a}(x)=\int_{F+a} f(x-a)\,d\mu(x)$$
and
$$\int f\,dT_a\mu|_{F+a}=\int f(x)h(x)\,d\mu(x).$$
Therefore the values of the function $f$ can be ignored outside $F$ and so we can assume $f$ is supported on $F$ and the same is true for $h$; and we have:
$$
\int_{F+a}f(x-a)\,d\mu(x)=\int_F f(x)h(x)\,d\mu(x).
$$

Since $M$ is the essential supremum of $h$, given $\epsilon>0$, we can find a subset $E$ of $F$, of positive measure $\mu$, such that $M-\epsilon\leq h\leq M$ on $E$.

Take $f_1:=\frac{1}{\sqrt{\mu(E)}}\chi_E$. We have $\|f_1\|_{L^2(\mu)}=1$. Also
$$\|f_1(\cdot-a)\|_{L^2(\mu)}^2=\int |f_1(x-a)|^2\,d\mu(x)=\int |f_1(x)|^2h(x)\,d\mu(x),$$
therefore $$M-\epsilon\leq \|f_1(\cdot-a)\|_{L^2(\mu)}^2\leq M.$$

We have
$$({f_1(\cdot-a)\,d\mu})^{\widehat{}}(t)=\int f_1(x-a)e^{-2\pi it\cdot x}\,d\mu(x)=e^{-2\pi i t\cdot a}\int f_1(x-a)e^{-2\pi it\cdot(x-a)}\,d\mu(x)$$$$=e^{-2\pi i t\cdot a}\int f_1(x)e^{-2\pi it\cdot x}h(x)\,d\mu(x)=e^{-2\pi i t\cdot a}\widehat{f_1h\,d\mu}(t)$$
This means that
$$|({f_1(\cdot-a)\,d\mu})^{\widehat{}}|=|\widehat{f_1h\,d\mu}|.$$

Next we estimate
$$\int|Mf_1-f_1h|^2\,d\mu=\int|f_1|^2|h-M|^2\,d\mu\leq \epsilon^2\|f_1\|_{L^2(\mu)}^2=\epsilon^2.$$
Then, using the upper frame bound:
$$\|\widehat{Mf_1\,d\mu}-\widehat{fh\,d\mu}\|_{L^2(\nu)}^2\leq B\|Mf_1-f_1h\|_{L^2(\mu)}^2\leq \epsilon^2B.$$
This implies that
$$\left|\|\widehat{Mf_1\,d\mu}\|_{L^2(\nu)}^2-\|(f_1(\cdot-a)\,d\mu)^{\widehat{}}\|_{L^2(\nu)}^2\right|=$$$$
\left|\|\widehat{Mf_1\,d\mu}\|_{L^2(\nu)}-\|(f_1(\cdot-a)\,d\mu)^{\widehat{}}\|_{L^2(\nu)}\right|\cdot\left(\|\widehat{Mf_1\,d\mu}\|_{L^2(\nu)}+\|(f_1(\cdot-a)\,d\mu)^{\widehat{}}\|_{L^2(\nu)}\right)$$
$$\leq \|\widehat{Mf_1\,d\mu}-\widehat{fh\,d\mu}\|_{L^2(\nu)}\cdot(\sqrt{B}M+\sqrt{B}\sqrt{M})\leq \epsilon\sqrt{B}(\sqrt{B}M+\sqrt{B}\sqrt{M})=:C\epsilon.$$

Then
$$\frac{\|\widehat{Mf_1\,d\mu}\|_{L^2(\nu)}^2}{\|(f_1(\cdot-a)\,d\mu)^{\widehat{}}\|_{L^2(\nu)}^2}=1+\frac{\|\widehat{Mf_1\,d\mu}\|_{L^2(\nu)}^2-\|(f_1(\cdot-a)\,d\mu)^{\widehat{}}\|_{L^2(\nu)}^2}{\|(f_1(\cdot-a)\,d\mu)^{\widehat{}}\|_{L^2(\nu)}^2}\leq 1+\frac{C\epsilon}{A(M-\epsilon)}.$$
On the other hand
$$\frac{\|\widehat{Mf_1\,d\mu}\|_{L^2(\nu)}^2}{\|(f_1(\cdot-a)\,d\mu)^{\widehat{}}\|_{L^2(\nu)}^2}\geq \frac{A\|Mf_1\|_{L^2(\mu)}^2}{B\|f_1(\cdot-a)\|_{L^2(\mu)}^2}\geq \frac{AM^2}{B(M-\epsilon)}.$$

Combining the two inequalities, and letting $\epsilon$ go to zero, we obtain
$$\frac{B}{A}\geq M$$
which is the desired inequality.

\end{proof}

\begin{definition}\label{def1.3}
Let $\mu$ be a Borel measure on $\br^d$. Let $f$ be a non-negative Borel measurable function. We define the {\it essential supremum} of $f$:
$${\esssup}_\mu(f)=\|f\|_\infty:=\inf\{ M\in[0,\infty] : f\leq M,\mu-\mbox{a.e.}\}.$$
We define the {\it essential infimum } of $f$:
$${\essinf}_\mu(f):=\sup\{ m\geq 0 : f\geq m,\mu-\mbox{a.e.}\}.$$
\end{definition}

\begin{theorem}\label{th1.3}
Let $\mu=g\,dx$ be an absolutely continuous measure on $\br^d$. If $\nu$ is a frame measure for $\mu$ with frame bounds $A,B>0$ then
$$\frac BA\geq \frac{\esssup_\mu(g)}{\essinf_\mu(g)}.$$

In particular, if $\esssup_\mu(g)=\infty$ or $\essinf_\mu(g)=0$ then there is no frame measure for $\mu$.
\end{theorem}

\begin{proof}
Let $M:=\esssup(g)$ and $m:=\essinf(g)$ and assume for the moment that $m>0$ and $M<\infty$. Take $\epsilon>0$ arbitrary. Then there exist a set of positive Lebesgue measure $C$ such that $m\leq g(x)\leq m+\epsilon$ for $x\in C$, and a set of positive Lebesgue measure $D$ such that $M-\epsilon\leq g(x)\leq M$ for all $x\in D$.

We need a lemma:

\begin{lemma}\label{lem1.4}
Let $C$ and $D$ be two sets of positive Lebesgue measure in $\br^d$. Then there exists a subset $E$ of $C$, of positive Lebesgue measure and some $a\in\br^d$ such that $E+a\subset D$.
\end{lemma}
\begin{proof}
Taking subsets we can assume $C$ and $D$ are bounded. Consider the convolution $\chi_C*\chi_{-D}$. We have
$$\chi_C\ast\chi_{-D}(t)=\int \chi_C(x)\chi_{-D}(t-x)\,dx=\int\chi_C(x)\chi_{D+t}(x)\,dx={\mathcal L}(C\cap(D+t)).$$
We claim that $\chi_C*\chi_{-D}$ cannot be identically zero. Taking the Fourier transform we have $\widehat{\chi_C*\chi_{-D}}=\widehat{\chi_C}\cdot\widehat{\chi_{-D}}$. Both functions are analytic in each variable and not identically zero. Hence their product cannot be identically zero. Therefore $\chi_C*\chi_{-D}(a)\neq 0$ for some $a\in\br^d$. So ${\mathcal L}(C\cap (D+a))>0$. Let $E:=C\cap (D+a)\subset C$. Then
$E-a\subset D$, and this proves the lemma.
\end{proof}

Returning to the proof of the theorem, using Lemma \ref{lem1.4} we find a set $E$ of positive Lebesgue measure and some $a\in\br^d$ such that $m\leq g(x)\leq m+\epsilon$ and $M-\epsilon\leq g(x+a)\leq M$ for all $x\in E$.

But the the measure $T_a(\mu|_{E+a})=g(x+a)\,dx|_E$ so
$$\left\|\frac{dT_a(\mu|_{E+a})}{d\mu}\right\|_\infty=\left\|\frac{g(x+a)}{g(x)}|_E\right\|_\infty\geq \frac{M-\epsilon}{m+\epsilon}.$$
Letting $\epsilon\rightarrow 0$ and using Theorem \ref{th1.2} we obtain the result.

Assume now $M=\infty$. Then for any $N$ we can find a subset $C$ of positive Lebesgue measure such that $N\leq \esssup(g|_C)<\infty$ and $0<\essinf(g|_C)\leq P$ for some fixed $P$.
Take the restriction $\mu|_C$ of the measure $\mu$ to $C$. Then it is clear that $\nu$ is also a frame measure for $\mu|_C$ with the same frame bounds. Then we can apply the previous argument to conclude that $B/A\geq N/P$. Letting $N\rightarrow\infty$ we obtain a contradiction. A similar argument shows that $\essinf(g)>0$.
\end{proof}

\medskip

We now give some corollaries of Theorem \ref{th1.3}. The first one is the case when $A=B$.
\begin{corollary}\label{cor1.5}
In the hypotheses of Theorem \ref{th1.3}. suppose $\mu=g\,dx$ admits a tight frame measure (or Plancherel measure), then $g$ is a constant multiple of a characteristic function.

In other words, if $g$ is not a constant multiple of a characteristic function then the measure $\mu=g\,dx$ does not admit tight frame measures; in particular it does not admit tight frames of weighted exponential functions $\{w_\lambda e_\lambda : \lambda\in\Lambda\}$, where $w_\lambda\in\bc$ for all $\lambda\in\Lambda$.
\end{corollary}

\begin{proof}
From Theorem \ref{th1.3}, we see that if $A=B$ then $\esssup_\mu(g)=\essinf_\mu(g)$ which means that $g$ is a characteristic function. The second statement follows by noting that weighted frames of exponentials correspond to discrete frame measures $\nu=\sum_{\lambda\in\Lambda}|w_\lambda|^2\delta_\lambda$, where $\delta_\lambda$ is the Dirac measure at $\lambda$.
\end{proof}

\medskip

If we replace the Lebesgue measure by general Hausdorff measure, we were not able to prove whether Theorem \ref{th1.3} will hold since Lemma \ref{lem1.4} cannot be generalized to Hausdorff measures; we have the following example.

\begin{example} Let $C$  be the set of
numbers in $[0,1]$ that can be represented in base 10 using digits
$\{0,1\}$ and $D$  be the same as $C$ except the digits are $\{0,2\}$. Then there is no $E\subset C$ with positive Hausdorff dimension (so none of its Hausdorff measures will be positive) such that $E+a\subset D$ for some $a\in{\Bbb R}$.
\end{example}

\begin{proof}
It is easy to see that $C-C$ is the set of numbers in $[-1/2,1/2]$ that have a base 10 representation with digits in $\{-1,0,1\}$, while $D-D$ is the set of numbers in $[-1/2,1/2]$ that have a base 10 representation with digits in $\{-2,0,2\}$. Hence, $(C-C)\cap(D-D)=\{0\}$.

Suppose there exists $E\subset C$ with positive Hausdorff dimension such that $E+a\subset D$ for some $a\in{\Bbb R}$. Then $E-E\subset C-C$, and $(E+a)-(E+a)\subset D-D$. But $E-E = (E+a)-(E+a)$, this implies that $E-E\subset (C-C)\cap (D-D)$. Hence, $E-E = \{0\}$. This means that $E$ has at most one point, so it has zero Hausdorff dimension. This is a contradiction.
\end{proof}

However, as Theorem \ref{th1.2} holds for general measures, we still have the following corollary.

\begin{corollary}\label{cor1.6}
Let $\H^s$ be the Hausdorff measure of dimension $s>0$ on $\br^d$. Let $d\mu=g(x)\,d\H^s(x)$ where $g$ is some non-negative Borel measurable function whose support $\Omega$ is a compact set with $0<\H^s(\Omega)<\infty$. Suppose there exists a Borel set $E$ and some $a\in\br^d$ such that $E,E+a\subset \Omega$ and such that there exist constants $0<m,M<\infty$ with
$$g(x)\leq m\mbox{ for all }x\in E\mbox{ and }g(x)\geq M\mbox{ for all }x\in E+a.$$
Then for any frame measure $\nu$ for $\mu$, its frame bounds $A,B$ satisfy the inequality
$$\frac BA\geq \frac Mm.$$
\end{corollary}

\begin{proof}
Since $\H^s$ is translation invariant, we have for $x\in E$:
$$\frac{dT_a(\mu|_{E+a})}{d\mu}(x)=\frac{g(x+a)}{g(x)}\geq \frac Mm.$$
The conclusion follows from Theorem \ref{th1.2}.
\end{proof}

\medskip

\section{Affine iterated function systems}

\begin{definition}\label{def4.1}
Let $R$ be a real $d\times d$ expansive matrix, i.e., all its eigenvalues $\lambda$ have absolute value $|\lambda|>1$. Let $\B = \{b_1,\cdots,b_N\}$ be a finite subset of $\br^d$ and let $(p_{b_i})_{i=1}^{N}$ be a set of positive probability weights, $p_{b_i}>0$ and $\sum_{i=1}^{N}p_{b_i}=1$. We define the {\it affine iterated function system}(IFS)
$$\tau_{b_i}(x):=R^{-1}(x+b_i),\quad (x\in \br^d, i = 1,\cdots N).$$
According to Hutchinson \cite{Hut81}, there exists a unique compact set $X_{\B}$ called the {\it attractor }  that has the invariance property
$$X_{\B}=\bigcup_{i=1}^{N}\tau_{b_i}(X_{\B}).$$
Moreover, in this case
\begin{equation}\label{eq4.1.1-}
X_{\B}=\left\{\sum_{n=1}^\infty R^{-n}b_n : b_n\in \B\mbox{ for all $n\in\bn$}\right\}.
\end{equation}
Also, there is a unique Borel probability measure $\mu_{\B}$ on $\br^d$ called the {\it invariant measure}, such that
\begin{equation}
\mu_{\B}(E)=\sum_{i=1}^{N}p_{b_i}\mu_{\B}(\tau_{b_i}^{-1}(E)),\quad \mbox{ for all Borel sets $E$}.
\label{eq4.1.1}
\end{equation}
In addition, the measure $\mu_{\B}$ is supported on the attractor $X_{\B}$. In this paper, we will write $X = X_{\B}$ and $\mu = \mu_{\B}$ when there is no confusion. Moreover, we will call the attractor and the invariant measure a {\it self-similar set} and a {\it self-similar measure} respectively if $R^{-1} =\lambda O$ for some $0<\lambda<1$ and orthogonal matrix $O$.

\medskip

If for all $i\neq j$, $i,j\in\{1,\cdots, N\}$, we have $\mu(\tau_{b_i}(X)\cap \tau_{b_j}(X))=0$ then we say that the affine IFS has {\it no overlap}.

\end{definition}

It is convenient to introduce some multiindex notation for a given affine IFS: let $\Sigma = \{1,\cdots, N\}$, $\Sigma^n= \underbrace{\Sigma\times\cdots\times\Sigma}_n$ and $\Sigma^{\ast} = \bigcup_{n=1}^{\infty}\Sigma^n$, the set of all finite words. Given $I= i_1\cdots i_n\in\Sigma^n$, $\tau_I(x) = \tau_{{b_{i_1}}}\circ\cdots \circ \tau_{{b_{i_n}}}(x)$, $p_I = p_{b_{i_1}}\cdots p_{b_{i_n}}$ and $X_I = \tau_I(X)$. By iterating the invariant identity of $X$, it is easy to see that
\begin{equation}\label{eq4.1.1+}
X = \bigcup_{I\in\Sigma^n}X_I.
\end{equation}
Finally, we write $I^n = \underbrace{I\cdots I}_n$ where $IJ$ denotes concatenation of the words $I$ and $J$. In this case, $\tau_{I^n}(x) = \tau_{I}\circ\cdots\tau_{I}(x) = \tau_{I}^{n}(x)$.

\medskip

We recall that, for the self-similar IFS, the well known \textit{open set condition}(OSC) states that there exists open set $U$ such that
$$\bigcup_{i=1}^{N}\tau_{b_i}(U)\subset U \ \mbox{and} \ \tau_{b_i}(U)\cap\tau_{b_j}(U) = \emptyset \ \mbox{for all} \ i\neq j.
$$
This condition is fundamental in fractal geometry. Before going to the main theorem in this section, we first clarify the relation between OSC and no overlap condition using theorems in \cite{S94} and \cite{LW93}.

\begin{proposition}\label{prop3.2}
If $\mu = \mu_{\B}$ is a self-similar measure, then the open set condition implies the no overlap condition of the measures $\mu$.
\end{proposition}
\begin{proof}
By \cite{S94}, we can choose an open set $U$ such that $U\cap X\neq\emptyset$. Pick $x\in U\cap X$, then there exists a ball of radius $\epsilon$ and centered at $x$, denoted by $B_{\epsilon}(x)$, is a subset of $U$. On the other hand, from (\ref{eq4.1.1+}) we have for all $n>0$ there exists some $I \in \Sigma^n$ such that  $x\in X_I$ (since $x\in X$). As the diameter of $X_I$ is tending to $0$ as $n$ tends to infinity, it shows that for $n$ large, $X_I\subset B_{\epsilon}(x)\subset U$. Writing $I = i_1\cdots i_n$, by iterating the invariance equation (\ref{eq4.1.1}),
$$
\mu(U)\geq \mu(X_I) \geq p_{b_1}\mu(X_{i_2\cdots i_n})\geq\cdots\geq p_{b_{i_1}}\cdots p_{b_{i_n}}>0.
$$
 We can then use Theorem 2.3 in \cite{LW93} to conclude that $\mu(U) =1$.  Writing also $U_b : = \tau_b(U)$ with $b\in\B$,  by Corollary 2.5 in \cite{LW93}, $\mu(\partial U_{b})=0$, where $\partial U_{b}$ denotes the boundary of $U_{b}$. As $X\subset \overline{U}$, the closure of $U$, we have $\tau_{b}(X)\subset \overline{U_b}$ and hence by $U_{b_i}\cap U_{b_j}=\emptyset$ from the OSC,
$$
\mu(\tau_{b_i}(X)\cap \tau_{b_j}(X))\leq\mu(\overline{U_{b_i}}\cap\overline{U_{b_j}}) = \mu_B(U_{b_i}\cap U_{b_j}) =0.
$$
\end{proof}

\begin{remark}
It is not known whether the no overlap condition implies the OSC. We know that the \textit{post-critically finite}(p.c.f.) fractals (the intersections consist only of finite points) introduced by Kigami \cite{Ki} satisfy the no overlap condition. However, except for some partial results in \cite{BR07} and \cite{DL08}, it is still an open question whether all p.c.f. fractals have the OSC.

Much less is known for affine iterated function system. We just know that if the OSC is satisfied, we can also choose $U$ to be an open set with non-empty intersection with the invariant set \cite{HeL08}. However, we do not know whether Proposition \ref{prop3.2} holds in affine IFS.
\end{remark}
\medskip

We can now prove the main theorem in this section using Theorem \ref{th1.2}.

\begin{theorem}\label{th4.2}
Let $(\tau_b)_{b\in \B}, (p_b)_{b\in \B}$ be an affine iterated function system with no overlap as in Definition \ref{def4.1}. Suppose the invariant measure $\mu$ admits a frame measure. Then all the probabilities $p_b$, $b\in \B$ must be equal.
\end{theorem}

\begin{proof}
The result will follow if we can check the absolute continuity assumption in Theorem \ref{th1.2}; this is given in the next lemma:
\begin{lemma}\label{lem4.3}
Pick two elements $b\neq c$ in $\B$ and let $n\in \bn$. Define $b^{(n)}:=b+Rb+\dots+R^{n-1}b$ and similarly for $c^{(n)}$. Let $a:=R^{-n}(c^{(n)}-b^{(n)})$ and  $F = \tau_b^n(X) $ (i.e. $\tau_b\circ\cdots\circ\tau_b(X)$ for $n$ compositions).

Consider the measure $T_a(\mu|_{F+a})$ with the notation as in Definition \ref{def1.1}. Then this measure is supported on $F$, it is absolutely continuous with respect to $\mu$ and the Radon-Nikodym derivative is constant on $F$:
$$\frac{dT_a(\mu|_{F+a})}{d\mu}=\frac{p_c^n}{p_b^n}.$$
\end{lemma}

\begin{proof}
It is easy to see that $\tau_b^n(x)=R^{-n}x+R^{-n}b^{(n)}$ and therefore $\tau_b^n(x)+a=\tau_c^n(x)$ for any $x\in\br^d$.
This implies that $F+a=\tau_c^n(X)$, so the measure $T_a(\mu|_{F+a})$ is supported on $\tau_b^n(X)$.
Also, we have $\tau_b^{-n}(x)=R^nx-b^{(n)}$.

For any $b$ in $\B$, we consider a arbitrary Borel set $E$ of $\tau_b^n(X)$. We note that $\tau_b^n(X)\subset \tau_b(X)$. By the fact that $\mu \circ\tau_{b'}^{-1}$ is supported on $\tau_{b'}(X)$, the no overlap condition and the invariance identity (\ref{eq4.1.1}), we get that for all $b'\neq b$,
 $$\mu(\tau_{b'}^{-1}(E))\leq \mu \circ\tau_{b'}^{-1}(\tau_b(X)\cap\tau_{b'}(X))\leq p_{b'}^{-1}\mu(\tau_b(X)\cap\tau_{b'}(X)) =0$$
  and hence for any $b$ in $B$
$$
\mu(E) = \sum_{b'\in \B}p_{b'}\mu (\tau_{b'}^{-1}(E))=p_b\mu (\tau_b^{-1}(E))=\cdots=p_b^n\mu(\tau_b^{-n}(E)).
$$

\medskip

Now for a Borel subset $E$ of $F$, we have that $E+a$ is contained in $F+a =\tau_c^n(X)$ and thus
$$T_a(\mu|_{F+a})(E)=\mu(E+a)=p_c^n\mu(\tau_c^{-n}(E+a))=
p_c^n\mu(R^n(E+a)-c^{(n)})=$$$$p_c^n\mu(R^nE-b^{(n)})=p_c^n\mu(\tau_b^{-n}(E))=
\frac{p_c^n}{p_b^n}\mu(E).$$
This establishes the absolute continuity and also that the density is exactly $p_{c}^{n}/p_b^{n}$.
\end{proof}

Returning to the proof of the theorem,  if we have a frame measure with frame bounds $A$ and $B$, then by Theorem \ref{th1.2} and  Lemma \ref{lem4.3}, we have that
$$\frac{B}{A}\geq \frac{p_b^n}{p_c^n}\mbox{ for all $b,c\in \B$ and $n\in\bn$}.$$
This implies that all the probabilities $p_b$ have to be equal.

\end{proof}

\bigskip

 In the remainder of this section, we focus on affine iterated function systems that do not satisfy the no overlap condition. We will prove some general results on ${\Bbb R}^d$ and then apply them to special cases in the next section. From the proof of Theorem \ref{th4.2}, we need to explore the following two questions:
\begin{enumerate}
\item Given any Borel measures $\mu$, is the measure $T_a(\mu|_{F+a})$ absolutely continuous with respect to $\mu$ for Borel sets $F$ in the support of $\mu$ with positive measure in $\mu$?
\item If $\mu = \mu_{\B}$, how to estimate $\mu(\tau_I(X))$?
\end{enumerate}

\medskip

In answering these questions, we found the results in \cite{HLW} particularly useful, for the case of self-similar invariant measures. Recall that ${\mathcal H}^{\alpha}$ denotes the $\alpha$-Hausdorff measure. We collect their results in the following theorem.

\begin{theorem}\label{th4.1}\cite{HLW} Let $\mu = \mu_{\B}$ be the self-similar measure defined in Definition \ref{def4.1}. Let $R = \lambda^{-1}O$ for some $0<\lambda<1$ and orthogonal matrix $O$. Then

(i) If $\mu\ll{\mathcal H}^{\alpha}|_{X}$, then ${\mathcal H}^{\alpha}|_{X}\ll\mu$.

(ii) If $\mu\ll {\mathcal L}|_{X}$ and the Radon-Nikodym derivative has an essential upper bound, then $p_{b_j}\leq \lambda^d$ for all $j$.
\end{theorem}

\medskip

For the first question,  when the measure is self-similar,  the following is a simple sufficient condition.
\begin{proposition}\label{prop3.3}
Suppose $\mu =\mu_{\B}$ is self-similar and $0<{\mathcal H}^{\alpha}(X)<\infty$. If the measure $\mu\ll{\mathcal H}^{\alpha}|_{X}$, then for any Borel sets $F$ in the support of $\mu$ and for any $a$, $T_a(\mu|_{F+a})\ll\mu$.
\end{proposition}

\begin{proof}
By Theorem \ref{th4.1}(i), ${\mathcal H}^{\alpha}|_{X}\ll\mu$ also.  Hence, if $E\subset F$ is a Borel set such that $\mu(E)=0$, then ${\mathcal H}^{\alpha}|_{X}(E)=0$. But $F\subset X$, so ${\mathcal H}^{\alpha}(E)={\mathcal H}^{\alpha}|_{X}(E)=0$. As the Hausdorff measure is invariant under translations, ${\mathcal H}^{\alpha}|_{X}(E+a)\leq{\mathcal H}^{\alpha}(E+a)=0$. Hence, by $\mu\ll{\mathcal H}^{\alpha}|_{X}$,
$$
T_a(\mu|_{F+a})(E) = \mu(E+a\cap F+a)\leq \mu(E+a) =0.
$$
But $T_a(\mu|_{F+a})$ is supported on $F$, hence we have established the absolute continuity.
\end{proof}

\medskip

The investigation of the second question is more difficult when there is overlap. For a self-affine measure in (\ref{eq4.1.1}), we can iterate the invariance identity $n$ times and then identify the maps $\tau_I, \tau_J$ such that $\tau_I = \tau_J :=\tau$. Denote by ${\mathcal A}_n$ the set all equivalence classes under this identification, for the compositions of $n$ maps that coincide, and let $p_{\tau}$ be the sum of the weights in the equivalence class (i.e., for $\tau\in{\mathcal A}_n$, $p_{\tau} = \sum\{p_I: \tau_I = \tau\}$). We therefore have
\begin{equation}\label{eq4.1.1++}
\mu = \sum_{\tau\in{\mathcal A}_n}p_{\tau}\mu\circ \tau^{-1}.
\end{equation}

\medskip

Note that if there is no overlap, ${\mathcal A}_n = \{\tau_I: I\in\Sigma^n\}$ and $p_{\tau_I} = p_I$. In this case, $\mu(\tau_{I}^n(X)) =p_{I}^n$. To extend our results to IFSs with overlap, we introduce the following definition.

 \begin{definition}\label{def5.1}
 Consider the IFS as in Definition \ref{def4.1}. Given $\tau\in{\mathcal A}_n$, we define $x_{\tau}$ to be the fixed point of $\tau$ if $x_{\tau} = \tau(x_{\tau})$. We say that the IFS satisfies the {\it fixed point condition} if there exists $k>0$ and $\tau\in{\mathcal A}_{k}$ such that the fixed point
 $$
 x_{\tau}\not\in \widetilde{\tau}(X) \  \mbox{for all} \ \widetilde{\tau}\neq\tau, \ \widetilde{\tau}\in{\mathcal A}_k.
 $$
 \end{definition}

 \medskip

 The following proposition shows that fixed point condition gives a partial answer to the second question.

\begin{proposition}\label{Prop4.1}
Given an IFS and suppose that the fixed point condition is satisfied for some $k\in {\Bbb N}$ and $\tau\in{\mathcal A}_k$. Then there exists $n_0$ such that for all $n\geq n_0$,
$$
\mu(\tau^n(X)) =  Cp_{\tau}^n
$$
 where $C$ is independent of $n$.
\end{proposition}

\begin{proof}  Writing $\tau = \tau_I$ for some $I =i_1\cdots i_k\in{\Sigma}_k$, $b_I= b_{i_k}+\cdots + R^{k-1}b_{i_1}$ and since $x_{\tau}= \tau_I(x_{\tau}) = \tau_I^n(x_{\tau})$ for all $n$, we have
$$
x_{\tau} = \sum_{n=1}^{\infty}R^{-kn}b_I.
$$
By (\ref{eq4.1.1-}), $x_{\tau}\in X$. Moreover,  $x_{\tau}= \tau^n(x_{\tau})\in\tau^n(X)$ for all $n\in{\Bbb N}$.  Since $\tau^n(X)$ and $\widetilde{\tau} (X)$, $\widetilde\tau\in \mathcal A_k$, are compact sets and the diameter of $\tau^n(X)$ tends to $0$, from the fixed point condition, there exists $n_0$ such that for all $n\geq n_0$, $\tau^n(X)\cap \widetilde{\tau}(X) = \emptyset$ for all $\widetilde{\tau}\neq \tau$ and $\widetilde{\tau}\in {\mathcal A}_k$.

 For all $n\geq n_0$, using the invariance identity (\ref{eq4.1.1++}),
$$
\mu(\tau^n(X)) = \sum_{\tau'\in{\mathcal A}_k}p_{\tau'}\mu (\tau'^{-1}(\tau^n(X))).
$$
From the above, we have $\mu (\tau'^{-1}(\tau^n(X))) =0$ if $\tau' \neq \tau$. Hence,
$$
\mu(\tau^n(X)) = p_{\tau}\mu (\tau^{-1}(\tau^n(X))) =  p_{\tau}\mu (\tau^{n-1}(X)).
$$
Inductively, $\mu(\tau^n(X)) = p_{\tau}^{n-n_0}\mu (\tau^{n_0}(X))= Cp_{\tau}^n $, where $C =p_{\tau}^{-n_0}\mu (\tau^{n_0}(X))$ is independent of $n$.
\end{proof}

\medskip

If we assume that the invariant measure is self-similar and is absolutely continuous with respect to the Lebesgue measure, we can use Theorem \ref{th0.1} and Proposition \ref{Prop4.1} to obtain the following:

\begin{theorem}\label{Prop4.2}
Let $\mu$ be a self-similar measure which is absolutely continuous with respect to the Lebesgue measure supported on $X$. If $\mu$ admits a frame measure then $p_{\tau} \leq \lambda^{dk}$ for all $\tau\in{\mathcal A}_k$. Suppose furthermore that the fixed point condition is satisfied for some $k\in {\Bbb N}$ and $\tau\in{\mathcal A}_{k}$, then, for these particular $k$ and $\tau$, $p_{\tau}= \lambda^{dk}$.

\end{theorem}

\begin{proof} Since $\mu\ll{\mathcal L}|_{X}$, we can write $d\mu = g(x)dx$ with $g$ is supported on $X$. As the measure is absolutely continuous with respect to the Lebesgue measure, ${\mathcal L}(X)>0$. By Theorem \ref{th0.1}, $g$ has finite, positive essential upper and lower bounds on $X$. By Theorem \ref{th4.1}(ii) applied to (\ref{eq4.1.1++}) (where the expanding matrix now becomes $\lambda^{k} O^k$), $p_{\tau} \leq \lambda^{dk}$ for all $j$. We now establish  $p_{\tau} \geq \lambda^{dk}$.

\medskip

Suppose now fixed point condition is satisfied but $p_{\tau} < \lambda^{dk}$. By Proposition \ref{Prop4.1}, we will then have for all $n\geq n_0$ with $n_0$ defined in Proposition \ref{Prop4.1} that,
$$
\frac{\mu(\tau^n(X))}{{\mathcal L}(\tau^n(X))} = \frac{\mu(\tau^n(X))}{\lambda^{dkn}{\mathcal L}(X)}\leq C\left(\frac{p_{\tau}}{\lambda^{dk}}\right)^n\rightarrow 0 \ {\mbox{as}} n\rightarrow \infty.
$$
But the density $g$ has a positive essential lower bound $m>0$, so
$$
\frac{\mu(\tau^n(X))}{{\mathcal L}(\tau^n(X))}=\frac{1}{{\mathcal L}(\tau^n(X))}\int_{\tau^n(X)}g(x)dx\geq m>0.
$$
This is a contradiction. Hence, $p_{\tau} = \lambda^{dk}$. This completes the proof.

\end{proof}

If the fixed point condition is satisfied by words which contain all the digits, then {\it all } the probability weights are equal.

\begin{corollary}\label{coradd2.1}
Let $\mu=\mu_\B$ be a self-similar measure which is absolutely continuous with respect to the Lebesgue measure supported on $X$ and assume that $\mu$ has a frame measure.
Suppose there exists a word $I$ in $\Sigma^n$ such that $I$ contains all the digits in $\{1,\dots,N\}$ and such that the fixed point $x_I$ of the map $\tau_I$ does not belong to any of the sets $\tau_J(X)$ for $J\in\Sigma^n$, $J\neq I$.
Then all the probabilities $p_{b_i}$ are equal, $\lambda^d=\frac1N$, there is no overlap and $\frac{d\mu}{d\mathcal L}=\frac{1}{\mathcal L(X)}\chi_X$.
\end{corollary}

\begin{proof}
The condition on $x_I$ given in the hypothesis implies that the only word $J$ for which $\tau_J=\tau_I$ is $J=I$. So, if $I=i_1\dots i_n$ then $p_{\tau_I}=p_I=\prod_{k=1}^np_{b_{i_k}}$.
From Theorem \ref{Prop4.2}, we have that $p_{b_i}\leq \lambda^d$ for all $i\in\{1,\dots,N\}$. Also, since the fixed point condition is satisfied for $n$ and $\tau_I$, we get that $p_{\tau_I}=\lambda^{dn}$. But then
$$\lambda^{dn}=\prod_{k=1}^n p_{b_{i_k}}\leq \lambda^{dn}.$$
This implies that $p_{b_{i_k}}=\lambda^d$ for all $k\in\{1,\dots,n\}$. Since all the digits in $\{1,\dots,N\}$ appear among the elements $\{i_1,\dots, i_N\}$ we obtain that all the probabilities $p_{b_i}$ are equal to $\lambda^d$.
Since they sum up to 1, this implies that $\lambda^d=\frac1N$. Since also ${\mathcal L}(X)$ is positive, $X$ is a self-similar tile \cite{LgW1} on ${\Bbb R}^d$. The rest of the statements will then follow. On the other hand, we can also prove it directly.

Since ${\mathcal L}(X)$ is positive, we apply the Lebesgue measure to the invariance identity of the attractor to get
$$\mathcal L(X)\leq \sum_{i=1}^N\mathcal L(\tau_{b_i}(X))=N\lambda^{d}\mathcal L(X)=\mathcal L(X).$$
This implies that the sets $\tau_{b_i}(X)$ have overlap of zero Lebesgue measure. Since $\mu$ is absolutely continuous, this means that the IFS has no overlap. We can then check that the Lebesgue measure on $X$, rescaled by $\frac{1}{\mathcal L(X)}$ to get a probability measure, is invariant for the IFS. By the uniqueness of the invariant measure we get that $\frac{d\mu}{d\mathcal L}=\frac{1}{\mathcal L(X)}\chi_{X}$. \end{proof}

\begin{remark}
We are not sure whether for any affine IFS there are always fixed points that satisfy the conditions in Theorem \ref{Prop4.2} or Corollary \ref{coradd2.1}. However, the fixed point conditions for finite iterations can many times be checked in concrete situations by an algorithm. In the next section, we will also see that there are always such fixed points for IFSs on ${\Bbb R}^1$.
\end{remark}

\medskip

\section{Iterated function systems on ${\Bbb R}^1$}

\medskip

We now apply the previous results to some IFSs with overlap to determine whether they have frame measures. Although these results can be applied on ${\Bbb R}^d$, we restrict our attention to ${\Bbb R}^1$ and there is no loss of generality to consider, upon rescaling, IFSs with functions $\tau_{b_i}(x) = \lambda x+b_i$, for $0<\lambda<1$, $i=1,..,N$ and
$$
\B =\{0=b_1<...<b_{N}=1-\lambda\}.
$$
In this case, the self-similar set  $X$ is a subset $[0,1]$. The self-similar measure with weights $p_i$ is the unique Borel probability measure satisfying
\begin{equation}\label{eq4.1.2}
\mu= \sum_{i=1}^{N}p_i\mu\circ\tau_{b_i}^{-1}.
\end{equation}

\begin{theorem}\label{th4.3}
Suppose the measure $\mu$ defined in (\ref{eq4.1.2}) is absolutely continuous with respect to ${\mathcal H}^{\alpha}|_{X}$ and $0<{\mathcal H}^{\alpha}(X)<\infty$. Then

\vspace{0.2 cm} {\rm (i)} If $\mu$ admits a frame measure, then $p_1 = p_N$.

\vspace{0.2 cm} {\rm (ii)} If $\alpha =1$ (i.e. $\mu\ll{\mathcal L}|_{X}$) and $\mu$ admits a frame measure, then $p_j\leq \lambda$ for all $j$ and $p_1 =p_N=\lambda$.


\end{theorem}

\begin{proof}
(i) Note that $\tau_{1}(0) = 0$, so the fixed point of $\tau_1$ is $0$. On the other hand, $\tau_{b_i}(X)\subset [b_i,b_i+\lambda]$. Hence, the fixed point condition holds for $I = 1$. Proposition \ref{Prop4.1} implies that there exists $n_0$ such that
\begin{equation}\label{eq4.2}
\mu(\tau_{1^n}(X)) =  C_1p_{1}^n, \ \mbox{for all } \ n\geq n_0.
\end{equation}
 Similarly, as  $\tau_{N}(1) =1$, we have
 \begin{equation}\label{eq4.3}
\mu(\tau_{N^n}(X)) =  C_2p_{N}^n, \ \mbox{for all } \ n\geq n_0.
\end{equation}

 Now for any $n\geq n_0$, define $F = \tau_{1}^{n}(X)$ and $a = 1-\lambda^n$. Then $F+a = \tau_{N^n}(X)$. By Proposition \ref{prop3.3}, $T_a(\mu|_{F+a})\ll\mu$. Let $h = dT_a(\mu|_{F+a})/d\mu$. Then by Theorem \ref{th1.2} and \eqref{eq4.2},
$$
T_a(\mu|_{F+a})(\tau_{1^n}(X)) = \int_{\tau_{1^n}(X)} hd\mu \leq \frac{B}{A}\mu(\tau_{1^n}(X))= \frac{C_1B}{A}p_1^n.
$$
On the other hand, $F+a = \tau_{N^n}(X)$ and so
$$
T_a(\mu|_{F+a})(\tau_{1^n}(X)) = \mu ( \tau_{N^n}(X)) = C_2p_N^n.
$$
Combining these, we obtain for all $n$,
$$
(\frac{p_N}{p_1})^n\leq\frac{C_1B}{C_2A}.
$$
This is possible only if $p_N\leq p_1$. By reversing the role of $1$ and $N$ and letting $a = -(1-\lambda^n)$, we obtain $p_1\leq p_N$.

To prove (ii), from (i) and the given assumption, all the conditions in Theorem \ref{Prop4.2} are satisfied. We have $p_1 = p_N = \lambda$.


\end{proof}

Roughly speaking, for an absolutely continuous self-similar measure that admits a frame measure, near the boundary, the measure must behave like a Lebesgue measure and this can only happen when $p_1 = p_N=\lambda$. In the middle part of the attractor, there are overlaps and we cannot conclude whether the weights are equal to $\lambda$.

\medskip

However, if now the measure admits a tight frame measure, then Corollary \ref{cor1.5} applies and we can actually solve the {\L}aba-Wang conjecture \cite{MR1929508} when the measure is absolutely continuous.

\begin{theorem}\label{thm4.4}
Suppose $\mu$ defined in (\ref{eq4.1.2}) is absolutely continuous with respect to the Lebesgue measure on $X$ and suppose $\mu$ admits a tight frame measure. Then

\vspace{0.2 cm} {\rm (i)} $p_1=\cdots=p_N=\lambda$.

\vspace{0.2 cm} {\rm (ii)} $\lambda=\frac{1}{N}$.

\vspace{0.2 cm} {\rm (iii)} There exists $\alpha>0$ such that $\D: = \alpha\B\subset{\Bbb Z}$ and $\D$ tiles ${\Bbb Z}$.

\end{theorem}

\begin{proof}
 Since $\mu=g\,dx$ has a tight frame measure, from Corollary \ref{cor1.5}, we have that $g$ is a multiple of a characteristic function, and since $\mu$ is supported on the attractor $X_\B$, we have that $\mu = c{\mathcal L}|_{X}$ for some constant $c>0$.

\medskip

We will prove by induction that $p_k=\lambda$ and $\tau_{b_k}(X_\B)\cap \tau_{b_{\ell}}(X_B)$ has Lebesgue measure zero for all $\ell>k$. When $k=1$, we know from  Theorem \ref{th4.3} that $p_1=\lambda$. From the invariance equation of $\mu$,
$$
\mu(\tau_{b_1}(X)) = \lambda\mu(X)+\sum_{j=2}^{N}p_j\mu(\tau_{b_j}^{-1}(\tau_{b_1}(X))).
$$
But $\mu(\tau_{b_1}(X)) =c{\mathcal L}|_{X}(\tau_{b_1}(X))=\lambda\mu(X)$, so the equation above implies that $\mu(\tau_{b_j}^{-1}(\tau_{b_1}(X)))=0$. In particular, this shows for all $j\geq2$, $
\mu(\tau_{b_1}(X))\cap\tau_{b_j}(X))=0.$
But since $\mu$ is a renormalized the Lebesgue measure on $X$, this proves the statement for $k=1$.
\medskip

Suppose we have proved the statement for all $i\leq k-1$. We now consider the set $A_k:=\tau_{b_k}\tau_{b_1}^n(X)$, where $n$ will be chosen later. This has positive Lebesgue measure and is contained in $X$ so it has positive $\mu$ measure. From the rescaling we considered, we have $X\subset [0,1]$.  We have that, for $l>k$, (recall that $b_1=0$),
$$
A_k\cap\tau_{b_{\ell}}X\subset (\lambda^{n+1}[0,1]+b_k)\cap (b_{\ell}+\lambda[0,1]).
$$
Since $b_k<b_{\ell}$ for ${\ell}>k$, we can pick $n$ large enough so that this intersection is empty. In this case, $\mu(\tau_{b_{\ell}}^{-1}(A_k))=0$. On the other hand,  for all $i\leq k-1$, by the induction hypothesis, $\mu(\tau_{b_i}^{-1}(A_k))\leq \mu(\tau_{b_i}^{-1}(\tau_{b_k}(X)))=0$. In the invariance equation, we have only the $k$-th term left:
$$
\mu(A_k) = p_k\mu(\tau_{b_k}^{-1}(A_k)).
$$
Again, $\mu$ is just the Lebesgue measure, so $p_k=\lambda$. Finally, using the induction hypothesis,
$$
\mu(\tau_{b_k}(X)) = \lambda\mu(X)+\sum_{\ell=k+1}^{N}p_{\ell}\mu(\tau_{b_{\ell}}^{-1}(\tau_{b_k}(X))).
$$
The no overlap follows in the same way as in $k=1$.

\medskip

By induction, we have proved (i). (ii) follows immediately from (i). Finally, we now have $\lambda^{-1}=\#\B$ and the attractor $X$ has positive Lebesgue measure. It means that the attractor is a self-similar tile on ${\Bbb R}^1$ \cite{LgW1}. By Theorem 4 in \cite{LgW2}, there exists $\alpha>0$ such that
$$
{\mathcal D} = \alpha{\mathcal B}\subset {\Bbb Z}.
$$
To prove that ${\mathcal D}$ tiles the integer lattice, we use some known properties of self-similar tiles. We will finish the proof in Proposition \ref{prop5.1} below.
\end{proof}

\medskip

Consider ${\mathcal D}\subset {\Bbb Z}$ and $\#{\mathcal D} = N$. Then if the attractor $X(=X(N,{\mathcal D}))$ of the IFS defined by $\tau_{d_i}(x) = N^{-1}(x+d_j)$ has positive Lebesgue measure, $X$ is a translational tile. A {\it self-replicating tiling set} of $X$ is a tiling set for $X$ which satisfies
 \begin{equation}\label{eqadd1.3}
{\mathcal J} = N{\mathcal J}\oplus {\mathcal D}.
\end{equation}
The direct sum here means that every element $t$ in ${\mathcal J}$ can be expressed uniquely as $Nt'+d$ for $t\in{\mathcal J}$ and $d\in{\mathcal D}$.

\begin{proposition}\label{prop5.1}
Let $\tau_{d_i}(x) = \frac{1}{N}(x+d_i)$, with ${\mathcal D}:=\{d_{i}\}\subset{\Bbb Z}$ and $\#{\mathcal D} =N$. If the attractor $X$ of $\{\tau_{d_i}\}$ is a self-similar tile on ${\Bbb R}^1$, then there exists ${\mathcal E}\subset {\Bbb Z}$ such that ${\mathcal D}\oplus{\mathcal E}={\Bbb Z}$.
\end{proposition}

\begin{proof}
This result actually holds for any dimension \cite{Lai12} by some deeper considerations from the theory of the self-affine tiles. Here, we give another proof in dimension 1 for completeness.

By translation and rescaling, we can assume ${\mathcal D}\subset{\Bbb Z}^{+}$, $0\in{\mathcal D}$ and g.c.d.${\mathcal D}=1$. From Theorem 3.1 in \cite{LR}, there exists a unique self-replicating tiling set ${\mathcal J}$ that is a subset of ${\Bbb Z}$ (i.e. ${\mathcal J}\subset {\Bbb Z}$).
In the following, we claim that there exists ${\mathcal G}$ such that ${\mathcal J}\oplus{\mathcal G} = {\Bbb Z}$. The proof is the similar to the proof of Theorem 3 in \cite{LgW2}.

\medskip

For $t\in[0,1)$ and a finite subset ${\mathcal G}$ in ${\Bbb Z}$, let
\begin{equation}\label{eq5.01}
{\mathcal G}(t) := \{j\in{\Bbb Z}: t+j\in X\} \ \mbox{and} \ X_{{\mathcal G}}: = \{t\in[0,1):{\mathcal G}(t) ={\mathcal G} \}.
\end{equation}
Since $X$ is compact, ${\mathcal G}(t)$ is a finite set and  only finitely many $X_{{\mathcal G}}$ are non-empty. Denote these non-empty sets by ${\mathcal G}_1,\cdots,{\mathcal G}_m$, then from the definitions in (\ref{eq5.01}),
$$
[0,1) = \bigcup_{j=1}^{m}X_{{\mathcal G}_j} \ \mbox{and} \ X = \bigcup_{j=1}^{m}\left(X_{{\mathcal G}_j}+{\mathcal G}_j\right).
$$
Moreover, $X_{{\mathcal G}_j}$ are mutually disjoint.
Thus $\{X_{\mathcal G_j}+k : j\in\{1,\dots,m\},k\in\bz\}$ is a partition of $\br$; also, since $X$ tiles $\br$ by $\mathcal J$, this implies that $\{X_{\mathcal G_j}+\mathcal G_j+k : j\in\{1,\dots,m\},k\in \mathcal J\}$ is a partition of $\br$. Then, for any $j$, the set $X_{{\mathcal G}_j}+{\mathcal G}_j$ tiles $X_{{\mathcal G}_j}+{\Bbb Z}$  using ${\mathcal J}$. Hence,
$$
X_{{\mathcal G}_j}+{\Bbb Z} = X_{{\mathcal G}_j}+{\mathcal G}_j\oplus {\mathcal J}.
$$
This shows that ${\Bbb Z} = {\mathcal G}_j\oplus{\mathcal J}$.

\medskip

Add ${\mathcal G} (={\mathcal G}_j) $ to both sides of (\ref{eqadd1.3}),
$$
{\Bbb Z} = {\mathcal J}\oplus{\mathcal G}  = N{\mathcal J}\oplus{\mathcal G} \oplus{\mathcal D}.
$$
This means that ${\mathcal D}$ tiles ${\Bbb Z}$ by ${\mathcal E}: =N{\mathcal J}\oplus{\mathcal G}$.
\end{proof}

\medskip

 In the end of this section, we apply our results to IFSs with a small number of maps. The simplest ones are the Bernoulli convolutions.

\begin{example}
Let us consider the biased Bernoulli convolution $\mu = \mu_{p,\lambda}$ with contraction ratio $0<\lambda<1$ as follows:
$$
\mu = p\mu\circ\tau_1^{-1}+(1-p)\mu\circ\tau_2^{-1}
$$
where $\tau_1(x) =\lambda x$ and $\tau_2(x)=\lambda x+(1-\lambda)$. Let also
 $$
 A = \{(p,\lambda):\mu_{p,\lambda}\ll{\mathcal L}\} \mbox{ and } S = (0,1)^2\setminus A.
 $$
Denote ${\mathcal F}  = \{(p,\lambda):\mu_{p,\lambda} \mbox{ has a frame measure} \}$. It is known that $\{(1/2,1/2n):n\in{\Bbb N}\}$ is contained in ${\mathcal F}$. We are interested in the question whether these are all the possible elements in ${\mathcal F}$. Concluding from the above theorems, we have
\begin{enumerate}
\item If $0<\lambda\leq1/2$, then the IFS satisfies the open set condition and hence has no overlap. This means that if $\mu$ has a frame measure, then $p=1/2$ by Theorem \ref{th4.2}.
\item If $1/2<\lambda<1$, there is non-trivial overlap. In this case, $p=1/2 = \lambda$ by Theorem \ref{th4.3}. Hence, we conclude that $A\cap {\mathcal F} = \{(1/2,1/2)\}$.

\end{enumerate}
\end{example}

\medskip

\begin{example}
{\rm The purpose of this example is to show how Theorem \ref{Prop4.2} can be applied to the sets $\tau_I(X)$, so that we can check if more general measures $\mu$ have a frame measure. Let }
$$
\tau_1(x) = \frac{1}{3}x, \ \tau_2(x) = \frac{1}{3}x+\frac{4}{21}, \ \tau_3(x) = \frac{1}{3}x+\frac{10}{21}, \ \tau_4(x) = \frac{1}{3}x+\frac{2}{3}
$$
{\rm and consider the self-similar measure $\mu$ defined as follows:}
\begin{equation}\label{eq4.4}
\mu = \frac{1}{3}\mu\circ \tau_{1}^{-1}+\frac{1}{6}\mu\circ \tau_{2}^{-1}+\frac{1}{6}\mu\circ \tau_{3}^{-1}+\frac{1}{3}\mu\circ \tau_{4}^{-1}.
\end{equation}
{\rm Then $\mu$ is absolutely continuous with respect to the Lebesgue measure on $[0,1]$, but $\mu$  has no frame measure.}
\end{example}

\begin{proof}

We can rescale the digit of the IFS by a factor $7/2$ so that the IFS becomes
$$
\tau_1(x) = \frac{1}{3}x, \ \tau_2(x) = \frac{1}{3}(x+2), \ \tau_3(x) = \frac{1}{3}(x+5), \ \tau_4(x) = \frac{1}{3}(x+7).
$$
The absolute continuity is completely determined by its mask polynomial
$$
m(\xi) = \frac{1}{3}+\frac{1}{6}e^{2\pi i 2\xi}+\frac{1}{6}e^{2\pi i 5\xi}+\frac{1}{6}e^{2\pi i 7\xi} = \frac{1}{6}(2+e^{2\pi i 2\xi}+e^{2\pi i 5\xi}+2e^{2\pi i 7\xi}).
$$
We note that $\mu$ is absolutely continuous if for all $n\in{\Bbb Z}\setminus\{0\}$, there exists $k$ such that $m(3^{-k}n)=0$ (see \cite[Theorem 1.1]{DFW}) The coefficients $c_i$ in this theorem will be $c_i=Np_i$ where $p_i$ are our probabilities and $N=3$ is the scaling factor, $\lambda=3$ and $d_i=b_i$ in the notation of \cite{DFW}. If $g$ is a solution to the refinement equation in \cite{DFW} then $\mu=g\,dx$ is our invariant measure). To check this condition, write $n = \pm3^{r}s$ for some $r\geq0$ and $3$ does not divide $s$. Let $k = r+1$, then $3^{-k}n = \pm s/3$. This implies that
$$
\begin{aligned}
m(3^{-k}n) =& \frac{1}{6}(2+e^{2\pi i 2s/3}+e^{2\pi i 5s/3}+2e^{2\pi i 7s/3})\\
 =& \frac{1}{3}(1+e^{2\pi i s/3} +e^{2\pi i 2 s/3}) =0. \  (\mbox{since 3 does not divide $s$})
\end{aligned}
$$

To see there is no frame measure, we note that we cannot use Theorem \ref{th4.3} since $p_1 =p_4 = \frac{1}{3}$ and probability weights are not equal. Now, we iterate (\ref{eq4.4}) one more time so that $\mu$ is the invariant measure of the IFS with the following $16$ maps (i.e. ${\mathcal A}_{2} = \{\tau_{ij}:i,j\in\{1,2,3,4\}\}$):
$$
\begin{array}{cccc}
  \tau_{11}(x) = \frac{1}{9}x & \tau_{12}(x) = \frac{1}{9}x+\frac{4}{63} & \tau_{13}(x) = \frac{1}{9}x+\frac{10}{63} & \tau_{14}(x) = \frac{1}{9}x+\frac{2}{9} \\
  \tau_{21}(x) = \frac{1}{9}x+\frac{4}{21} & \tau_{22}(x) = \frac{1}{9}x+\frac{16}{63} & \tau_{23}(x) = \frac{1}{9}x+\frac{22}{63} & \tau_{24}(x) = \frac{1}{9}x+\frac{26}{63} \\
  \tau_{31}(x) = \frac{1}{9}x+\frac{10}{21} & \tau_{32}(x) = \frac{1}{9}x+\frac{34}{63} & \tau_{33}(x) = \frac{1}{9}x+\frac{40}{63} & \tau_{34}(x) = \frac{1}{9}x+\frac{44}{63} \\
  \tau_{41}(x) = \frac{1}{9}x+\frac{2}{3} & \tau_{42}(x) = \frac{1}{9}x+\frac{46}{63} & \tau_{43}(x) = \frac{1}{9}x+\frac{52}{63} & \tau_{44}(x) = \frac{1}{9}x+\frac{56}{63}
\end{array}
$$
and the weight for $\tau_{ij}$ is $p_ip_j$. Moreover, it is easy to see that the self-similar set $X$ of this IFS is $[0,1]$. Consider $\tau_{23}(x)$, the fixed point $x_{23}=\frac{11}{28}$. Note that the map that can overlap with $\tau_{23}(X)$ are $\tau_{22}(X)$ and $\tau_{24}(X)$. Since $\tau_{22}(X) = [16/63,23/63]$ and $\tau_{24}(X) = [26/63,31/63]$, a direct calculation shows that $x_{23}$ is not in $\tau_{22}(X)$ nor in $\tau_{24}(X)$.
 Since also $\tau_{23}(X)\cap \tau_{ij}(X)=\emptyset$ for all other $ij\neq 22$ or $24$,  $x_{23}$ does not belong to all the other $\tau_{ij}(X)$. In particular, if $\mu$ has a frame measure, then Theorem \ref{Prop4.2} applies which shows that $p_2p_3 = \lambda^2=\frac{1}{9}$, but this is not the case since $p_2p_3 = \frac{1}{36}$.
\end{proof}

\medskip

\section{Concluding remarks on frame measures}

 The study of frame measures or Fourier frames for singular measures is intriguing and leaves a lot of open problems for us to investigate. In the following, we outline the strategies and problems which may be essential towards a full solution for the case of singular measures.

\medskip

The main strategy exhibited in this paper is based on the the assumption that  measures restricted on a subset are absolutely continuous after translations of that subset. We don't know whether measures with a frame measure must satisfy this assumption.
However, there do exist examples for which such translational absolute continuity fails. The following suggests that singular measures supported essentially on positive Lebesgue measurable sets give such examples.

\begin{example}\label{example3.1}
{\rm Let $\mu$ be a measure whose support is exactly $[0,1]$. Suppose $\mu$ is singular with respect to the Lebesgue measure on $[0,1]$, then there exists $F, F+a\subset[0,1]$ such that $T_a(\mu|_{F+a})$ is singular with respect to $\mu$.}
\end{example}

\begin{proof}
As the measure is singular with respect to the Lebesgue measure on $[0,1]$, we can find a set $E\subset [0,1]$ such that ${\mathcal L}(E)>0$ but $\mu(E)=0$. By decomposing $[0,1]$ into dyadic intervals, we may assume $E$ is in some dyadic interval $F =[i2^{-n},(i+1)2^{-n}]$ for any $n$. Let $I = \{x: F+x\subset [0,1]\} = [-i2^{-n},1-(i+1)2^{-n}]$. Note that
$$
\int_{I}\mu(E+x)dx = \int\int_{I}\chi_{E+x}(y)dxd\mu(y) = \int\int_{I}\chi_{y-E}(x)dxd\mu(y) =\int{\mathcal L}((y-E)\cap I)d\mu(y).
$$
As $-E\subset[-(i+1)2^{-n},-i2^{-n}]$, we have that $y-E\subset I$ if $y\in[2^{-n},1-2^{-n}]$.
$$
\int_{I}\mu(E+x)dx\geq\int_{2^{-n}}^{1-2^{-n}}{\mathcal L}(E)d\mu(y) ={\mathcal L}(E)\mu([2^{-n},1-2^{-n}])>0.
$$
Here, $\mu([2^{-n},1-2^{-n}])>0$ because $\mu$ is supported on $[0,1]$. Hence, there exists $a$ such that $\mu(E+a)>0$. To complete the proof, we note that $\mu(E)=0$ but $T_a(\mu|_{F+a})(E) = \mu(E+a\cap F+a) = \mu(E+a)>0$, this shows the singularity of the measures.
\end{proof}

There are many measures that satisfy the condition in Example \ref{example3.1}. In the case of self-similar measures, one of the most common examples are the  Bernoulli convolutions with overlap and with contraction ratio equal to a Pisot number \cite{PSS}.

\medskip

To the best of our knowledge, measures that have a frame measure should distribute mass on the support in a uniform way. It is natural to expect that assumption in Theorem \ref{th1.2} should be necessary for the existence of frame measures. In particular, we say that a finite Borel measure $\mu$ is {\it translationally absolutely continuous} if for all Borel sets $F$ in the support of $\mu$ and $\mu(F)>0$ and for all $a\in{\Bbb R}^d$, $T_a\mu|_{F+a}\ll\mu$.

Another concept that describes, for a given measure $\mu$, the differences in its local distribution is the {\it local dimension} at points $x\in\mbox{supp}(\mu)$. Let $\alpha>0$
$$
K(\alpha): = \{x\in\mbox{supp}\mu: \mbox{dim}_{loc}\mu(x):= \lim_{r\rightarrow 0}\frac{\log \mu(B_r(x))}{\log r} \ \mbox{exists and equals} \ \alpha\}
$$
where $B_{r}(x)$ is the ball of radius $r$ centered at $x$. If $x\in K(\alpha)$, it means that for all $\epsilon>0$, we have for all $r$ sufficiently small,
 $$
r^{\alpha-\epsilon}\leq\mu(B_r(x))\leq r^{\alpha+\epsilon}.
 $$
The standard  $1/n$-Cantor measure $\mu_n$ has only one local dimension $\log2/\log n$. If $\mu_n$ is convolved with a discrete measure of finite number of atoms, it still has only one local dimension. On the other hand,  it is known that equal contraction non-overlapping self-similar measures with unequal probability weights have more than one local dimension. We do not know examples of  measures that have more than one local dimension and that have a frame measure. Heuristically, if there are two local dimensions, the balls around two points scale differently which means the mass around those balls is not evenly distributed. Combining these observations, we propose the following conjecture:

\medskip

\begin{conjecture}
If $\mu$ is a measure with a frame measure, then $\mu$ must be translationally absolutely continuous and it has only one local dimension.
\end{conjecture}

\medskip

In other words, such a measure has only trivial multifractal structure. On the other hand, even if the $\mu$ has only one local dimension, we still need to classify the measures for which there exists a Fourier frame. In particular, the following is a famous problem:
\medskip

\noindent{\bf (Q1):} Does the one-third Cantor measure have a frame measure, Fourier frame or exponential Riesz basis?

\medskip

It is known that the middle third Cantor measure has no orthogonal spectrum. For some recent approaches, in \cite{DHSW10},  necessary conditions  for the existence of frame spectrum are found in terms of the \textit{Beurling dimension}. It is also shown that all fractal measures arising from the iterated function systems with equal contraction ratios admit some Bessel exponential sequences of positive Beurling dimension \cite{DHW11a}. However, there is still no complete answer to the question. The standard one-third Cantor measure is a measure with only one local dimension $\log 2/\log3$, so the method we used in this paper cannot work. While we contend that it is difficult to answer whether the Cantor measure has frame measures or not, we can ask the following simpler questions:

\medskip

\noindent {\bf (Q2)}. Find a singular measure with a Fourier frame but which is not absolutely continuous with respect to a spectral measure nor a convolution of spectral measures with some discrete measures.

\medskip

\noindent {\bf (Q3)} Find a self-similar measure admitting a Fourier frame of the type described in Q2.

\medskip

\noindent {\bf (Q4)} If a measure has a frame measure, does it have a Fourier frame?

\medskip

\section{An application: Gabor orthonormal bases}
In this section, we consider the Gabor system of the form
$$
{\mathcal G}(g,\Lambda,{\mathcal J}) := \{e^{2\pi i \lambda\cdot x}g(x-p): \lambda\in\Lambda, \ p\in{\mathcal J}\}
$$
where $g\in L^2({\mathbb R}^d)$, $\Lambda$, ${\mathcal J}$ are discrete sets in ${\mathbb R}^d$. We say that $\mathcal G(g,\Lambda,\mathcal J)$ is an {\it orthonormal basis} if the functions in the system ${\mathcal G}(g,\Lambda,{\mathcal J})$ is orthonormal and for all $f\in  L^2({\mathbb R}^d)$,
\begin{equation}\label{eq2.1}
\sum_{\lambda\in\Lambda}\sum_{p\in{\mathcal F}}|\int f(x)e^{-2\pi i \lambda\cdot x}\overline{g(x-p)}dx|^2 = \|f\|_2^2.
\end{equation}
 We also observe that if ${\mathcal G}(g,\Lambda,{\mathcal J})$ is a Gabor orthonormal basis of $L^2({\mathbb R}^d)$, then for any $(\lambda_0,p_0)\in{\mathbb R}^{2d}$, ${\mathcal G}(g,\Lambda-\lambda_0,{\mathcal J}-p_0)$ is also a Gabor orthonormal basis of $L^2({\mathbb R}^d)$. Hence, there is no loss of generality to assume $(0,0)\in\Lambda\times{\mathcal J}$.
%

   We recall one proposition due to Jorgensen and Pedersen.

\begin{proposition}\label{thm1.3}
\cite{JP98} Let $\mu$ be a compactly supported probability measure on ${\Bbb R}^d$. Then $\{e_{\lambda}:\lambda\in\Lambda\}$ is an orthonormal basis on $L^2(\mu)$ if and only if
$$
\sum_{\lambda\in\Lambda}|\widehat{\mu}(x+\lambda)|^2\equiv1.
$$
\end{proposition}

%
%

We will now prove the conjecture in \cite{LW} when $g$ is non-negative. Our main theorem is as follows,

\begin{theorem}\label{thm3.1}
Let  $g\in L^2({\mathbb R}^d)$ be non-negative function supported on a bounded set $\Omega$ with positive Lebesgue measure. Let $\Lambda$ and ${\mathcal J}$ be discrete sets. Suppose that ${\mathcal G}(g,\Lambda,{\mathcal J})$ is a Gabor orthonormal basis of $L^2({\mathbb R}^d)$, then

{\rm (i)} ${\mathcal J}$ is a tiling set of $\Omega$.

{\rm (ii)} $|g(x)| = \frac{1}{\sqrt{{\mathcal L}(\Omega)}}\chi_{\Omega}(x)$ a.e. on $\Omega$.

{\rm (iii)} $\{e_{\lambda}:\lambda\in\Lambda\}$ is a spectrum of $L^2(\Omega)$.

\end{theorem}

\medskip

\begin{proof}
We divide the proof into three claims.

\medskip

\textit{Claim 1: If ${\mathcal G}(g,\Lambda,{\mathcal J})$ is complete in $L^2({\mathbb R}^d)$, then ${\mathcal L}({\mathbb R}^d\setminus\bigcup_{p\in{\mathcal J}}(\Omega+p)) =0$.}

\medskip

 Suppose that ${\mathcal L}({\mathbb R}^d\setminus\bigcup_{t\in{\mathcal J}}(\Omega+t)) >0$, let $K\subset {\mathbb R}^d\setminus\bigcup_{t\in{\mathcal J}}(\Omega+t)$,  be such that $0<{\mathcal L}(K)<\infty$. Then $f =\chi_K$, then $f$ is a nonzero $L^2$ function, but
$$
\int f(x)e^{2\pi i \lambda\cdot x}g(x-p)dx =0
$$
since $g(\cdot-p)$ is supported on $\Omega+p$ which is disjoint from $K$. This contradicts the completeness of the system.

\medskip

\textit{Claim 2: If ${\mathcal G}(g,\Lambda,{\mathcal J})$ is a Gabor orthonormal basis in $L^2({\mathbb R}^d)$, then ${\mathcal L}((\Omega+p)\cap(\Omega+p')) =0$, for all $p\neq p'$ and $p,p'\in{\mathcal J}$.}

\medskip


Suppose for some $p\neq p'$, we have ${\mathcal L}(\Omega_{p,p'})>0$ where $\Omega_{p,p'} = (\Omega+p)\cap(\Omega+p')$. By the orthonormality of the functions represented by $(0,p)$ and $(0,p')$, we have
$$
\int_{\Omega_{p,p'}}g(x-p)g(x-p')dx =0.
$$
As $g$ is non-negative, $g(\cdot-p)g(\cdot-p') =0$ almost everywhere on $\Omega_{p,p'}$. But $g(\cdot-p)$ and $g(\cdot-p')$ are supported on $\Omega+p$ and $\Omega+p'$ respectively and they are non-zero almost everywhere there. This is a contradiction since $\Omega_{p,p'}$ has positive Lebesgue measure.
\medskip

\textit{Claim 3: $\{e_{\lambda}:\lambda\in\Lambda\}$ is a spectrum of $L^2(|g|^2dx)$.}

\medskip

For any $t\in{\Bbb R}^d$, we let $f_t(x) = g(x)e^{2\pi i \langle t,x\rangle}$. Then $\int|f_t|^2 = \int|g|^2<\infty$. We use this in (\ref{eq2.1}) and obtain
$$
\sum_{\lambda\in\Lambda}|\int|g(x)|^2e^{2\pi i (t-\lambda)\cdot x}dx|^2+\sum_{\lambda\in\Lambda}\sum_{p\in{\mathcal J}\setminus\{0\}}|\int g(x)g(x-p)e^{2\pi i (t-\lambda)\cdot x}dx|^2 = \int|g(x)|^2dx=1
$$
where $\int|g|^2=1$ follows from the orthonormality and $(0,0)\in\Lambda\times{\mathcal J}$.
As $g(\cdot)g(\cdot-p)$ is non-zero only on $\Omega\cap\Omega+p$ which is of Lebesgue measure 0 by claim 2, we get that  $g(\cdot)g(\cdot-p)=0$ almost everywhere and thus all the integrals in the second sum on the left hand side are zero. Hence,
\begin{equation}\label{eq3.1}
\sum_{\lambda\in\Lambda}|(|g|^2dx)^{\widehat{}}(t-\lambda)|^2 \equiv \int|f|^2d\mu.
\end{equation}
This is equivalent to say $\Lambda$ is a spectrum of $L^2(|g|^2dx)$ by Proposition \ref{thm1.3}.

 \medskip

 We can now complete the proof the theorem. Claim 1 and 2 shows that ${\mathcal J}$ is a tiling set of $\Omega$. This proves (i).  By Corollary \ref{cor1.5} and claim 3, $|g| = c\chi_{\Omega}$. As  $\int|g|^2dx=1$ and we can see easily that $c = ({\mathcal L}(\Omega))^{-1/2}$. Hence (ii) holds. Finally (iii) follows immediately from claim 3.
\end{proof}

\medskip

\noindent{\bf Acknowledgment.} The second author would like to thank Professor Ka-Sing Lau for his teaching and guidance over the years. He would also like to thank Professor De-Jun Feng and Professor Chi-Wai Leung for valuable and inspiring discussions.

\bibliographystyle{alpha}
\bibliography{eframes}

\begin{thebibliography}{DHW11b}

\bibitem[BR07]{BR07}
Christopher Bandt and Hui Rao.
\newblock Topology and separation of self-similar fractals in the plane.
\newblock {\em Nonlinearity}, 20:1463--1474, 2007.

\bibitem[DFW07]{DFW}
Xin-Rong Dai, De-Jun Feng, and Yang Wang.
\newblock Refinable functions with non-integer dilations.
\newblock {\em J. Func. Anal.}, 250:1--20, 2007.

\bibitem[DHJ09]{DHJ09}
Dorin~Ervin Dutkay, Deguang Han, and Palle E.~T. Jorgensen.
\newblock Orthogonal exponentials, translations and bohr completions.
\newblock {\em J.Funct. Anal.}, 257:2999--3019, 2009.

\bibitem[DHS09]{DHS09}
Dorin~Ervin Dutkay, Deguang Han, and Qiyu Sun.
\newblock On the spectra of a {C}antor measure.
\newblock {\em Adv. Math.}, 221(1):251--276, 2009.

\bibitem[DHSW11]{DHSW10}
Dorin~Ervin Dutkay, Deguang Han, Qiyu Sun, and Eric Weber.
\newblock On the {B}eurling dimension of exponential frames.
\newblock {\em Adv. Math.}, 226:285--297, 2011.

\bibitem[DHW11a]{DHW11a}
Dorin~Ervin Dutkay, Deguang Han, and Eric Weber.
\newblock Bessel sequences of exponentials on fractal measures.
\newblock {\em J. Functional Anal.}, 261(9):2529--2539, 2011.

\bibitem[DHW11b]{DHW11b}
Dorin~Ervin Dutkay, Deguang Han, and Eric Weber.
\newblock Continuous and discrete {F}ourier frames for fractal measures.
\newblock {\em preprint}, 2011.

\bibitem[DL08]{DL08}
Qi-Rong Deng and Ka-Sing Lau.
\newblock Open set condition and post-critically finite self-similar sets.
\newblock {\em Nonlinearity}, 21:1227--1232, 2008.

\bibitem[DS52]{DS52a}
R.~Duffin and A.~Schaeffer.
\newblock A class of nonharmonic {F}ourier series.
\newblock {\em Trans. Amer. Math. Soc.}, 72:341--366, 1952.

\bibitem[Gro00]{G00}
Karlheinz Grochenig.
\newblock {\em Foundation of Time-Frequency Analysis}.
\newblock Applied and Numerical Harmonic Analysis. Birkh\"auser Boston Inc.,
  Boston, MA, 2000.

\bibitem[HL08a]{HeL08}
Xing-Gang He and Ka-Sing Lau.
\newblock On a generalized dimension of self-affine fractals.
\newblock {\em Math. Nachr.}, 281:1142--1158, 2008.

\bibitem[HL08b]{HL08}
Tian-You Hu and Ka-Sing Lau.
\newblock Spectral property of the {B}ernoulli convolutions.
\newblock {\em Adv. Math.}, 219(2):554--567, 2008.

\bibitem[HLL11]{HLL11}
Xing-Gang He, Ka-Sing Lau, and Chun-Kit Lai.
\newblock Exponential spectra in {L}$^2(\mu)$.
\newblock {\em preprint}, 2011.

\bibitem[HLW01]{HLW}
Tian-You Hu, Ka-Sing Lau, and Xiang-Yang Wang.
\newblock Absolute continuity of a class of invariant measures.
\newblock {\em Proc. Amer. Math. Soc.}, 130:759--767, 2001.

\bibitem[Hut81]{Hut81}
John~E. Hutchinson.
\newblock Fractals and self-similarity.
\newblock {\em Indiana Univ. Math. J.}, 30(5):713--747, 1981.

\bibitem[IP00]{MR1744572}
Alex Iosevich and Steen Pedersen.
\newblock How large are the spectral gaps?
\newblock {\em Pacific J. Math.}, 192(2):307--314, 2000.

\bibitem[JKS07]{MR2338387}
Palle E.~T. Jorgensen, Keri~A. Kornelson, and Karen~L. Shuman.
\newblock Affine systems: asymptotics at infinity for fractal measures.
\newblock {\em Acta Appl. Math.}, 98(3):181--222, 2007.

\bibitem[JP98]{JP98}
Palle E.~T. Jorgensen and Steen Pedersen.
\newblock Dense analytic subspaces in fractal {$L\sp 2$}-spaces.
\newblock {\em J. Anal. Math.}, 75:185--228, 1998.

\bibitem[Kig01]{Ki}
Jun Kigami.
\newblock {\em Analysis on Fractals}.
\newblock Cambridge Tracts in Mathematics, vol. 143, Cambridge University
  Press, Cambridge, 2001.

\bibitem[Lai11]{Lai11}
Chun-Kit Lai.
\newblock On {F}ourier frame of absolutely continuous measures.
\newblock {\em J.Funct. Anal.}, 261:2877–--2889, 2011.

\bibitem[Lai12]{Lai12}
Chun-Kit Lai.
\newblock Spectral analysis on fractal tiles and measures.
\newblock {\em PhD thesis, CUHK}, 2012.

\bibitem[Li07]{MR2297038}
Jian-Lin Li.
\newblock {$\mu\sb {M,D}$}-orthogonality and compatible pair.
\newblock {\em J. Funct. Anal.}, 244(2):628--638, 2007.

\bibitem[LR03]{LR}
Ka-Sing Lau and Hui Rao.
\newblock On one-dimensional self-similar tilings and the $pq$-tilings.
\newblock {\em Tran. of Amer. Math. Soc.}, 355:1401--1414, 2003.

\bibitem[LW93]{LW93}
Ka-Sing Lau and Jian-rong Wang.
\newblock Mean quadratic variations and {F}ourier asymptotics of self-similar
  measures.
\newblock {\em Monatshefte Math}, 115:99--132, 1993.

\bibitem[LW96a]{LgW1}
J.C. Lagarias and Yang Wang.
\newblock Self-affine tiles in {${\mathbb R}^n$}.
\newblock {\em Adv. in Math.}, 121:21--49, 1996.

\bibitem[LW96b]{LgW2}
J.C. Lagarias and Yang Wang.
\newblock Tiling the line with translates of one tile.
\newblock {\em Invent. Math.}, 124:341--365, 1996.

\bibitem[{\L}W02]{MR1929508}
Izabella {\L}aba and Yang Wang.
\newblock On spectral {C}antor measures.
\newblock {\em J. Funct. Anal.}, 193(2):409--420, 2002.

\bibitem[LW03]{LW}
Youming Liu and Yang Wang.
\newblock The uniformity of non-uniform {G}abor bases.
\newblock {\em Adv. Comput. Math.}, 18(2-4):345--355, 2003.
\newblock Frames.

\bibitem[{\L}W06]{MR2200934}
Izabella {\L}aba and Yang Wang.
\newblock Some properties of spectral measures.
\newblock {\em Appl. Comput. Harmon. Anal.}, 20(1):149--157, 2006.

\bibitem[OCS02]{OSANN}
Joaquim Ortega-Cerd{\`a} and Kristian Seip.
\newblock Fourier frames.
\newblock {\em Ann. of Math. (2)}, 155(3):789--806, 2002.

\bibitem[PSS00]{PSS}
Yuval Peres, Wilhelm Schlag, and Boris Solomyak.
\newblock {\em Sixty years of {B}ernoulli convolutions}.
\newblock Fractals and Stochastics II, (C. Bandt, S. Graf and M. Zaehle, eds),
  Progress in probability, 46, Birhauser, 2000.

\bibitem[Sch94]{S94}
Andreas Schief.
\newblock Separation properties for self-similar sets.
\newblock {\em Proc. Amer. Math. Soc}, 122:111--115, 1994.

\bibitem[Str00]{MR1785282}
Robert~S. Strichartz.
\newblock Mock {F}ourier series and transforms associated with certain {C}antor
  measures.
\newblock {\em J. Anal. Math.}, 81:209--238, 2000.

\bibitem[Str06]{MR2279556}
Robert~S. Strichartz.
\newblock Convergence of mock {F}ourier series.
\newblock {\em J. Anal. Math.}, 99:333--353, 2006.

\bibitem[Yua08]{MR2443273}
Yan-Bo Yuan.
\newblock Analysis of {$\mu\sb {R,D}$}-orthogonality in affine iterated
  function systems.
\newblock {\em Acta Appl. Math.}, 104(2):151--159, 2008.

\end{thebibliography}

\end{document}